\documentclass[letterpaper, 10 pt, journal, twoside]{IEEEtran}  

\IEEEoverridecommandlockouts                              

\usepackage{amsthm} 
\newtheorem{thm}{Theorem}
\newtheorem{defn}{Definition}
\newtheorem{corr}{Corollary}

\newtheorem{rem}{Remark}
\newtheorem{assum}{Assumption}

\newtheorem{prob}{Problem}

\newtheorem{innercustomthm}{Problem}
\newenvironment{customthm}[1]
  {\renewcommand\theinnercustomthm{#1}\innercustomthm}
  {\endinnercustomthm}
\usepackage{caption}
\usepackage{subcaption}
\usepackage{hyperref}
\usepackage{cite}
\usepackage{multirow}
\usepackage[export]{adjustbox}
\usepackage{amssymb} 
\usepackage{amsbsy}
\usepackage{amsmath} 
\usepackage{mathtools}
\usepackage{cases}
\usepackage{nccmath}
\usepackage{optidef}
\usepackage{mathrsfs}
\usepackage{upgreek}
\usepackage{adjustbox}
\usepackage{tikz}
\usepackage{bigdelim}
\usepackage{soul}
\usepackage{xcolor}

\usetikzlibrary{shapes,arrows, intersections, fit}

\usepackage{algorithm}
\usepackage{algpseudocode}

\algdef{SE}[DOWHILE]{Do}{doWhile}{\algorithmicdo}[1]{\algorithmicwhile\ #1}%

\makeatletter
\define@key{Gin}{Trim}
            {\let\Gin@viewport@code\Gin@trim\expandafter\Gread@parse@vp#1 \\}
\makeatother

\usepackage{array}
\newcolumntype{M}[1]{>{\centering\arraybackslash}m{#1}}

\begin{document}
\title{Fast, Convexified Stochastic Optimal Open-Loop Control For Linear Systems Using Empirical Characteristic Functions}

\author{Vignesh Sivaramakrishnan and Meeko M. K. Oishi
\thanks{
    This material is based upon work supported by the National Science Foundation 
    under NSF Grant Number IIS-1528047 and CNS-1836900,
    and by the Air Force Research Lab under Grant No. FA9453-18-2-0022. 
    Any opinions, findings, and conclusions or recommendations expressed in this
    material are those of the authors and do not necessarily reflect the views
    of the National Science Foundation. \newline\indent 
    V. Sivaramakrishnan and M. Oishi are with Electrical \& Computer Engineering, University of New Mexico, Albuquerque, New Mexico. \newline\indent
    Email: {\tt\{vigsiv,oishi\}@unm.edu}.
}}
\date{}
\maketitle

\begin{abstract}
We consider the problem of stochastic optimal control in the presence of an unknown disturbance. 
We characterize the disturbance via empirical characteristic functions, and employ a chance constrained approach. 
By exploiting properties of characteristic functions and underapproximating cumulative distribution functions, we can reformulate a nonconvex problem by a conic, convex under-approximation.  
This results in extremely fast solutions that are assured to maintain probabilistic constraints. 
We construct algorithms to solve an optimal open-loop control problem
and demonstrate our approach on two examples.  
\end{abstract}

\section{Introduction}

Stochastic optimal control typically presumes accurate models of the underlying dynamics and stochastic processes \cite{mesbah2016stochastic,bertsekas2004stochastic,stengel1994optimal}.
However, in many circumstances, accurate characterization of uncertainty is difficult.  Further, inaccurate characterization of stochastic processes may have unexpected impacts \cite{nilim2005robust,samuelson2017data}, as optimal control actions are typically dependent upon the first and second moments of the stochastic processes \cite{stengel1994optimal}.  
Such inaccuracies could be particularly problematic when the unknown stochastic processes is asymmetric, multimodal, or heavy-tailed.  
For example, in hypersonic vehicles, excessive turbulence makes aerodynamic processes difficult to model accurately, and their fast time-scale means that erroneous control actions 
could result in catastrophic failure.  

We consider the case in which the dynamics are known, but the noise process is not known, and
focus on the problem of data-driven stochastic optimal control in a chance constrained setting, in which probabilistic constraints must be satisfied with at least a desired likelihood.
Some approaches, such as distributional stochastic optimal control, seek robustness to ill-defined distributions with finite samples \cite{samuelson2017data,yang_convex_2017}. 
Other approaches construct piecewise-affine over-approximations of value functions by  solving a chance-constrained problem~\cite{darivianakis2017data}. 
Researchers have also employed kernel density estimation \cite{calfa_data-driven_2015,caillau_solving_2018} to approximate individual chance constraints in nonlinear optimization problems.  

One tool to characterize uncertainty through observed data is the {\em empirical characteristic function} \cite{yu_empirical_2004}, which is often employed in economics and statistics to characterize models where maximum-likelihood estimation can struggle.  
The empirical characteristic function generates an approximation of the true characteristic function, and has known convergence properties \cite{csorgo_limit_1981,feuerverger_empirical_1977}.  
The advantage of this approach is that it enables direct, closed-form approximation of the cumulative distribution function and the moments of the underlying stochastic process \cite{yu_empirical_2004}, both of which are typically necessary for stochastic optimal control problems.  
However, the main challenge then becomes one of finding computationally efficient under-approximations of the resulting cumulative distribution function, which may be non-convex.  

We propose to employ empirical characteristic functions to characterize unknown disturbance processes in a linear, time-invariant dynamical system with a quadratic cost function.  
We construct a conic, convex reformulation of the resulting stochastic optimal control problem, that ensures computational tractability \cite{blackmore2010minimum}.  
Our approach employs a piecewise under-approximation of the approximate cumulative distribution function, with a user-specified trade-off between accuracy and the number of piecewise elements.  
We use confidence intervals on the approximate cumulative distribution function to provide probabilistic bounds on the solution to the data-driven stochastic optimal control problem.  
 \textit{The main contribution of this paper is the construction of a convex, conic reformulation of a stochastic optimal control problem in the presence of an unknown, additive disturbance, via empirical characteristic functions, with confidence bounds on the optimal solution.}
 
The outline of the paper is as follows. 
We first formulate the problem in Section~\ref{sec:prelimform}.  
Section~\ref{sec:method} presents algorithms to convexify the problem and proofs of its convergence properties. 
In Section~\ref{sec:numerical}, we demonstrate our approach on two examples. 

\section{Preliminaries and Problem formulation}\label{sec:prelimform}
We use the following notation throughout the paper.  
We denote real-valued vectors with lowercase $w \in \mathbb R^n$, matrices with upper case $V \in \mathbb R^{n \times m}$, and random variables via boldface $\mathbf{w}$.  Concatenated vectors or matrices are indicated by an overline, 
$    \overline{\mathbf{w}} = \left[\mathbf{w}[0]^{\top}\,\mathbf{w}[1]^{\top}\,\cdots\,\mathbf{w}[N-1]^{\top}\right]^{\top}\in\mathbb{R}^{pN}
$.  
We denote intervals using $\mathbb{N}_{[a,b]}$ where $a,b\in\mathbb{N},\; a<b$.

Consider the linear time-invariant dynamical system 
\begin{equation} 
    \mathbf{x}[k+1] = A\mathbf{x}[k] + B u[k] + G \mathbf{w}[k]
\end{equation}
with state $\mathbf{x} \in \mathbb R^n$, controlled input $u \in \mathbb R^m$, disturbance input $\mathbf{w} \in \mathbb R^p$, matrices $A, B, G$ of the appropriate dimensions, and timestep $k \in [0, N]$. 
Given a deterministic initial condition $x_0$, we rewrite the dynamics in concatenated form 
\begin{equation}\label{eq:concatdyn}
    \overline{\mathbf{x}} = \overline{A}x_0+\overline{B}\overline{u}+\overline{G}\overline{\mathbf{w}}
\end{equation}
with state $\overline{\mathbf{x}}\in\mathbb{R}^{n(N+1)}$, input $\overline{u}\in\mathcal{U}^{N} = [u_{\mathrm{min}}, u_{\mathrm{max}}]^{N}\subset\mathbb{R}^{mN}$, 
disturbance $\overline{\mathbf{w}}\in\mathbb{R}^{pN}$, and matrices $\overline A, \overline B, \overline G$,
as in \cite{cinquemani_convexity_2011,vitus_stochastic_2016}.

We presume $\mathbf{\overline{w}}$ is a stationary, independent stochastic process, that is the concatenation of a sequence of samples, $\left\{\mathbf{w}_j\right\}_{j=1}^{N_s}$, drawn from the probability space $\Omega$.  
The probability space is defined by $(\Omega,\mathscr{B}(\mathbb{R}^{pN}),\mathbb{P}_{\mathbf{\overline{w}}})$ with $\mathbb{P}_{\mathbf{\overline{w}}}$ as the induced probability distribution of $\mathbb{P}$ \cite[Prop. 2.1]{eaton1983multivariate}. 


\begin{prob}
Solve the optimization problem
\begin{subequations}
\label{prob:stoc}
\begin{align}
     & \underset{\overline{u}}{\mathrm{min}} &\;
     \mathbb{E}\Bigl[\left({\overline{\mathbf{x}}} -\overline{x}_d \right)^\top Q \left(\overline{\mathbf{x}} -\overline{x}_d \right) + \overline{u}^\top R & \overline u \Bigr]\label{eq:gencost}\\
     & \mathrm{s.t.} &  \mathbb{P}\left\{\overline{\mathbf{x}}\notin \mathcal{S}\right\} \leq&\  \Delta\label{eq:gensconstr}\\
     & & \overline{u}\in&\ \mathcal{U}^{N}\label{eq:genhardinput}
\end{align}
\end{subequations}
subject to the dynamics in (\ref{eq:concatdyn}), for a desired trajectory $\overline{x}_d\in\mathbb{R}^{n(N+1)}$,
positive definite matrices  $Q\in\mathbb{R}^{n(N+1)\times n(N+1)}$ and $R\in\mathbb{R}^{mN\times mN}$, polytopic constraint set $\mathcal S \subseteq \mathbb R^{n(N+1)}$ that is closed and bounded, and constraint violation threshold $\Delta\in [0,1]$, {\em without} direct knowledge of the cumulative distribution function or moments of $\mathbf{w}$, but {\em with} observations of $N_s$ samples $\left\{\mathbf{w}_j\right\}_{j=1}^{N_s}$.  
\end{prob}

The standard approach to solving (\ref{prob:stoc}) when the disturbance process is well characterized is to tighten the joint chance constraint (\ref{eq:gensconstr}) via individual chance constraints \cite{cinquemani_convexity_2011,vitus_stochastic_2016}.  
However, 
two main challenges then arise: 1) reliance of (\ref{eq:gencost}) and (\ref{eq:gensconstr}) upon moments and the cumulative distribution function, respectively, of the unknown noise process, and 2) non-convexity of the individual chance constraints.   
The former can be seen from expanding (\ref{eq:gencost}),
\begin{align}
    &\mathbb{E}\Bigl[\left({\overline{\mathbf{x}}} -\overline{x}_d \right)^\top Q \left(\overline{\mathbf{x}} -\overline{x}_d \right) + \overline{u}^\top R \overline u \Bigr] =\nonumber\\
    (&\mathbb{E}[\overline{\mathbf{x}}] - \overline{x}_d)^\top Q{(\mathbb{E}[\overline{\mathbf{x}}] -
    \overline{x}_d)}+\overline{u}^\top R\overline{u}+ \mathrm{tr}(Q\overline{G}\text{diag}(C_{\overline{\mathbf{w}}})\overline{G}^{\top})
\end{align}
with $\mathbb{E}[\overline{\mathbf{x}}]= \overline{A}x_0+\overline{B}\overline{u}+\overline{G}\mathbb{E}[\overline{\mathbf{w}}]$, $C_{\overline{\mathbf{w}}} = \mathbb{E}[\overline{\mathbf{w}}^2] - (\mathbb{E}[\overline{\mathbf{w}}])^2$.

Characteristic functions provide a means to obtain moments as well as the cumulative distribution function. 

\begin{defn}
The characteristic function of a random vector $\mathbf{w}\in\mathbb{R}^{p}$ is 
\begin{equation}
\label{eq:CF}
    \upvarphi_{\mathbf{w}}(t) = \mathbb{E}[\exp{(it^{\top}\mathbf{w})}]= 
    \int_{\mathbb{R}^p}~\exp{(it^{\top}\mathbf{w})}~d\Phi_{\mathbf{w}}(x)
\end{equation}%
which is the Riemann–Stieltjes integral of $\exp{(it^{\top}\mathbf{w})}$ over the frequency variable $t\in\mathbb{R}^{p}$ with respect to the cumulative distribution function, $\Phi_{\mathbf{w}}(x)$.
\end{defn}%

Since we have no direct knowledge of 
$\mathbf{w}$, the empirical characteristic function can be used to compute the cumulative distribution function and moments from samples of $\textbf{w}$.

\begin{defn}[Empirical Characteristic Function
{\cite{feuerverger_empirical_1977,yu_empirical_2004}}] Let $\left\{\mathbf{w}_j\right\}_{j=1}^{N_s}$ be the sequence of $N_s$ observations of the random vector, $\mathbf{w}$. The empirical characteristic function is 
\begin{subequations}
    \label{eq:ecf}
\begin{align}
\useshortskip
    \hat\upvarphi_{\mathbf{w}}(t) &=  \sum^{N_s}_{j=1}\alpha_j (\mathbf{w}) K_{\mathbf{w}_j}(t) \\
    K_{\mathbf{w}_j}(t) &= \exp{(it^{\top}\mathbf{w}_j)}\exp{\left(-\textstyle{ \frac{1}{2}}(t^{\top}\Sigma t)\right)}\label{eq:ecfsmoothing}
\end{align}
\end{subequations}
for some smoothing parameter matrix $\Sigma\in\mathbb{R}^{p\times p}$ and weighting
function $\alpha_j (\mathbf{w})>0$, with $\sum_{j=1}^{N_s}\alpha_j (\mathbf{w})=1$.
\label{def:ecf}
\end{defn}

A variety of approaches can be used to find a suitable $\Sigma$, to avoid over-smoothing and under-smoothing \cite{silverman_density_1998}.
The smoothing in \eqref{eq:ecfsmoothing}  is important for ensuring continuity in the cumulative distribution function \cite[Eq. 1.2.1]{lukacs_characteristic_1970} approximated via Theorem \ref{thm:GPI} from the empirical characteristic function.

\begin{thm}[Gil-Pileaz Inversion Theorem,\cite{gil-pelaez_note_1951}]
\label{thm:GPI}
The cumulative distribution function of a random variable $\boldsymbol{y}\in\mathbb{R}$ can be written in terms of the characteristic function as 
\begin{equation}
            \label{eq:cftocdf}
         \Phi_{\mathbf{y}}(x) = \frac{1}{2}-\frac{1}{\pi}\int^{\infty}_{0}\mathrm{Im}\left(\frac{\exp{(-it_{\mathbf{y}}x)}~\upvarphi_{\mathbf{y}}(t_{\mathbf{y}})}{t_{\mathbf{y}}}\right)~dt_{\mathbf{y}}
\end{equation}
where $t_{\mathbf{y}}\in\mathbb{R}$ and $\upvarphi_{\mathbf{y}}(t_{\mathbf{y}})$ is the characteristic function of the random variable $\mathbf{y}$. 
\end{thm}
The integral in \eqref{eq:cftocdf} is assured to converge, 
since it is a convex combination of characteristic functions  (Definition~\ref{def:ecf}), 
which exist for any random vector
\cite[Thm. 2.1.3]{lukacs_characteristic_1970}.

\begin{defn}
The $d^{\mathrm{th}}$ moment of $\mathbf{w}$ can be written as 
\begin{equation}\label{eq:momentdesired}
    \mathbb{E}[\mathbf{w}^d] = (-i)^{d} \left[\frac{\partial^d\upvarphi_{\mathbf{w}}(t)}{\partial t_1^d}\cdots\frac{\partial^d\upvarphi_{\mathbf{w}}(t)}{\partial t_p^d}\right]^{\top}_{\substack{t=0}}
\end{equation}
\end{defn}

Hence to solve Problem 1, we first solve the following. 
\begin{customthm}{1.a}\label{ps:ecf}
    Using the empirical characteristic function, 1) construct a concave under-approximation of the approximate cumulative distribution function $\hat\Phi_{\mathbf{w}}(x)$, and 2) approximate the first two moments of $\mathbf{\overline{w}}$.
\end{customthm}
\begin{customthm}{1.b}\label{ps:optimal_hp}
    Reformulate (3) into a convex, conic stochastic optimal control problem, so that feasible solutions of the convex program are feasible solutions of (3).
\end{customthm}

\section{Method}\label{sec:method}

We first transform (\ref{eq:gensconstr}) into a series of individual chance constraints, each with a risk $\delta_i$. 
We represent the set $\mathcal{S}$ as $\mathcal{S} = \{\overline{\mathbf{x}}\in\mathbb{R}^{n(N+1)}:P\overline{\mathbf{x}}\leq q\}$ for some $P \in\mathbb{R}^{l\times n(N+1)}$, $q\in\mathbb{R}^l$. 
Denoting the $i^{\mathrm{th}}$ constraint as $p_i^{\top} \overline{\mathbf{x}} \leq q_i$, we obtain
\begin{subequations}
\begin{align}
    \mathbb{P}\left\{p_i^{\top}G\overline{\mathbf{w}}\leq q_i-p_i^{\top}(\overline{A}x_0+\overline{B}\overline{u})\right\} &\geq 1 - \delta_i\label{eq:ccfinal}\\
    \Leftrightarrow
    \Phi_{p_i^{\top}G\overline{\mathbf{w}}}(q_i-p_i^{\top}(\overline{A}x_0+\overline{B}\overline{u})) &\geq 1 - \delta_i \label{eq:ICCcdf}\\
    \sum_{i=1}^{l} \delta_i \leq \Delta,\; \delta_i \geq 0,\; \Delta\in[0,1],\;\forall i&\in\ \mathbb{N}_{[1,l]}\label{eq:riskalloc}
\end{align}
\end{subequations}
for $p_i\in\mathbb{R}^{n(N+1)}$, $q_i \in \mathbb R$, $\delta_i \in [0,1] \subseteq \mathbb R$ with $\overline{\delta}\in\mathbb{R}^l$. 

Then solutions of the optimization problem 
\begin{subequations}
\label{prob:stoc_ecf_cr}
\begin{alignat}{3}
  \min_{\overline{u},\overline{\delta}} &\hspace{2.5em} \mathbb{E}\left[\left(\overline{{\mathbf{x}}} -\overline{x}_d \right)^\top  Q \left({\overline{\mathbf{x}}} -\overline{x}_d \right) + \overline{u}^\top  R \overline u \right]\label{eq:stoc_ecf_cr_cost}\\
  \mathrm{s.t.} 
  &
      \raisebox{-.6\baselineskip}[0pt][0pt]{\hspace{-1.3em}\footnotesize{$\forall i \in \mathbb{N}_{[1,l]}$}\hspace{0.4em}}
    \raisebox{-.5\baselineskip}[0pt][0pt]{\hspace{-0.5em}$\left\{\kern-\nulldelimiterspace\begin{array}{ @{} c } \mathstrut \\ \mathstrut \\ \mathstrut \end{array}\right.$}
  \hspace{-0.3em}\Phi_{p_i^{\top}\overline{G}\overline{\mathbf{w}}}(q_i-p_i^{\top}(\overline{A}x_0+\overline{B}\overline{u})) \geq\  1 - \delta_i \label{eq:stoc_ecf_cr_ICC} \\
  &\hspace{7.2em} q_i-p_i^{\top}(\overline{A}x_0+\overline{B}\overline{u}) \geq\ x^{lb}_{i}\label{eq:stoc_ecf_cr_ICC_res}\\
  &\hspace{6em} \sum_{i=1}^{l} \delta_i \leq \Delta,\; \delta_i \geq 0,\; \Delta\in\ [0,1]\label{eq:stoc_ecf_cr_ra}\\
  & \hspace{15em} \overline{u}\in\ \mathcal{U}^{N}\label{eq:stoc_ecf_cr_u}
\end{alignat}
\end{subequations}
are also feasible solutions of (\ref{prob:stoc}).
This is because the joint chance constraint \eqref{eq:gensconstr} is enforced by \eqref{eq:stoc_ecf_cr_ICC} and \eqref{eq:stoc_ecf_cr_ra} with the additional constraint \eqref{eq:stoc_ecf_cr_ICC_res}, which 
restricts the domain of the $i^{\mathrm{th}}$ chance constraint by some lower bound $x_i^{lb}$. 

However, several difficulties arise.  
Note that (\ref{prob:stoc_ecf_cr}) is non-convex due to \eqref{eq:stoc_ecf_cr_ICC}.  
The constraint (\ref{eq:stoc_ecf_cr_ICC_res}) ensures a restriction to the concave region of $\Phi_{p_i^{\top}\overline{G}\mathbf{\overline{w}}}(x)$. For unimodal distributions, the inflection point, $x_i^{lb}$, occurs about the mode\cite[Def. 1.1]{dharmadhikari1988unimodality}, but for arbitrary distributions, this may not be true. 

In addition, (\ref{eq:stoc_ecf_cr_cost}) is dependent upon the first two moments of $\overline{\mathbf{w}}$ and (\ref{eq:stoc_ecf_cr_ICC}) is dependent upon the cumulative distribution function of $p_i^{\top}\overline{G}\overline{\mathbf{w}},\ \forall i \in \mathbb{N}_{[1,l]}$.  
Hence we seek empirical characteristic functions to approximate the cumulative distribution function and moments based on samples $\mathbf{w}_j$. 
In addition, we also seek a method to reformulate (\ref{eq:stoc_ecf_cr_ICC}) using its approximation from the empirical characteristic function with a concave restriction \eqref{eq:stoc_ecf_cr_ICC_res} by finding $x_i^{lb}$ to solve a convex problem.

\subsection{Approximating the cumulative distribution function and moments from the empirical characteristic function}

Applying Definition~\ref{def:ecf}, we obtain
\begin{subequations}
\begin{align}
    \hat\upvarphi_{p_i^{\top}\overline{G}\overline{\mathbf{w}}}(t) =&\ \sum^{N_s}_{j=1}\alpha_j(\mathbf{\overline{w}})\exp{(itp_i^{\top}\overline{G}\overline{\mathbf{w}}_j)} \cdot & \nonumber\\
    & \hspace{1.2em} \exp{\left(-\tfrac{1}{2}((p_i^{\top}\overline{G})\overline{\Sigma}(p_i^{\top}\overline{G})^{\top} t^2)\right)}&\label{eq:cdfecfwkern}\\
        \hat\upvarphi_{\overline{\mathbf{w}}}(t) =&\ \sum^{N_s}_{j=1}\alpha_j(\mathbf{\overline{w}})\exp{(i\overline{t}^{\top}\overline{\mathbf{w}}_j)}\exp{\left(-\tfrac{1}{2}(\overline{t}^{\top}\overline{\Sigma}\overline{t})\right)}
        \label{eq:momentecfwkern}
\end{align}%
\end{subequations}
where $\overline{\Sigma} = \text{diag}([\Sigma_0\cdots\Sigma_{N}])\in\mathbb{R}^{pN\times pN}$ , $\overline t = [t_0\ \cdots t_{N}]^{\top}\in\mathbb{R}^{pN}$ and $\alpha_j(\mathbf{\overline{w}}) = 1/N_s$. 
To approximate $\Phi_{p_i^{\top}\overline{G}\overline{\mathbf{w}}}$ in \eqref{eq:stoc_ecf_cr_ICC}, we use \eqref{eq:cftocdf} 
to obtain $\hat\Phi_{p_i^{\top}\overline{G}\overline{\mathbf{w}}}$.
For the moments in the cost \eqref{eq:stoc_ecf_cr_cost}, we use \eqref{eq:momentdesired} to obtain the approximate moments of $\overline{\mathbf{w}}$.

\subsection{Constructing a Convex Restriction for~\eqref{eq:stoc_ecf_cr_ICC}}
\begin{figure*}[t!]
    \centering
    \includegraphics[width=\textwidth]{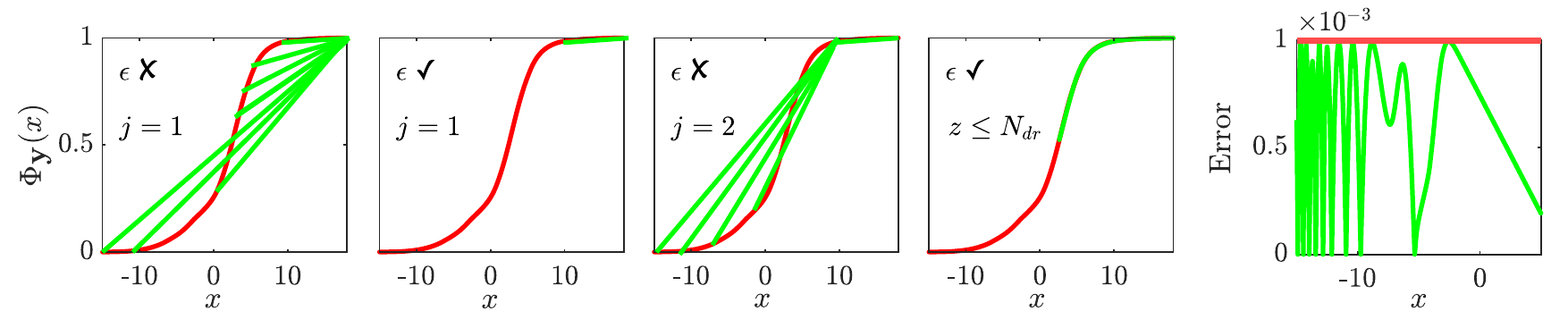}
    \caption{(Left to Right) Algorithm~\ref{algo:PWAunder} under-approximates the cumulative distribution function, $\hat\Phi_{\mathbf{y}}(x)$ (red), with $\hat\Phi^l_{\mathbf{y}}(x)$ (green), for some user-defined error, $\epsilon$.  
    We use 1000 samples of $\mathbf{y}=\mathbf{f}\mathbf{y}_1+(1-\mathbf{f})\mathbf{y}_2$, with Bernoulli random variable $\mathbf{f}$, $\mathbf{y}_1$ a Gaussian $\mathcal{N}(0,0.2)$, and $\mathbf{y}_2$ a Weibull distribution $\text{Weib}(k = 4,\theta = 2)$. 
    The error $\hat\Phi_{\mathbf{y}}(x)-\hat\Phi^l_{\mathbf{y}}(x) \leq \epsilon$ is depicted on the far right.}
    \label{fig:PWAunder}
\end{figure*}

We seek  a conic representation of~\eqref{eq:stoc_ecf_cr_ICC} with a restriction for which it is concave \cite[Def 1.1]{dharmadhikari1988unimodality}.
For a user-defined error, $\epsilon$, and desired number of affine terms, $N_{dr}$, we construct a piecewise affine under-approximation\cite[Sec. Sec. 4.3.1]{boyd2004convex}, 

\begin{align}\label{eq:PWACDF}
    \hat\Phi^l_{p_i^{\top}\overline{G}\overline{\mathbf{w}}} &= \min_{r\in \mathbb{N}_{[1,z^i]}}\{a_{i,r}x+c_{i,r}\}
\end{align}
such that 
\begin{equation}\label{eq:approxineq}
    0\leq \hat\Phi_{p_i^{\top}\overline{G}\overline{\mathbf{w}}}(x)
    -\hat\Phi^l_{p_i^{\top}\overline{G}\overline{\mathbf{w}}}(x) \leq \epsilon 
\end{equation}
is assured 
over the domain $\mathcal{D}_i = [x^{lb}_i,\infty]$.  
We define $a_{i,r}$ and $c_{i,r}$ as the slope and intercept for the $r^{\mathrm{th}}$ affine term. 

We propose Algorithm 1 to construct the piecewise linear under-approximation of the cumulative distribution function, with a concave restriction $x^{lb}$, derived from the empirical characteristic function. 

\begin{algorithm}
  \caption{Computing $\hat\Phi^l_{\mathbf{w}}$ from $\hat\Phi_{\mathbf{w}}$}
	\label{algo:PWAunder}
	Evaluations of cumulative distribution function $\{(x_p,\hat\Phi_{\mathbf{w}}(x_p)\}_{p=1}^{N_p}$, desired error $\epsilon$, desired number of affine terms $N_{dr}$.
	\\
	\textbf{Output}:
	affine terms of $\hat\Phi^l_{\mathbf{w}}$, $\{(a_{j},c_{j})\}^{z}_{j=1}$, restriction $x^{lb}$
	\begin{algorithmic}[1]
	\State continue $\leftarrow$ true, $p\leftarrow N_p$ 
		\While{continue = true} $\forall j\in\mathbb{N}_{[1,p-1]},\; \forall k\in\mathbb{N}_{[j,p]}$
		\State $a_j\leftarrow\frac{\hat\Phi_{\mathbf{w}}(x_{p})-\hat\Phi_{\mathbf{w}}(x_j)}{x_{p}-x_j}$
		\State $c_j\leftarrow \hat\Phi_{\mathbf{w}}(x_{p})-m_jx_{j}$
		\State $y_{j,k}\leftarrow a_jx_{k} + c_j$
		\State $error_{j,k}$ $\leftarrow \hat\Phi_{\mathbf{w}}(x_{k})-y_{j,k}$
		\State $w\leftarrow$ Smallest $j$ such that $\underset{j}{\max}\{error_{j,k}\}<\epsilon$ and $error_{j,k}>0$
		\If{$w = \emptyset$ or  $z>N_{dr}$ or $||w-p|| = 1$}
		\State continue $\leftarrow$ false
		\State $(a_{j},c_{j})\}^{z}_{j=1}\leftarrow$$\{(0,\Phi_{\mathbf{w}}(x_{N_p}))\}\bigcup\mathcal{F}$, $x^{lb}\leftarrow x_j$
		\Else, $\mathcal{F}\leftarrow\{(a_j,c_j)\}_{j=w}$, $p = w$
		\EndIf
		\EndWhile
  \end{algorithmic}
\end{algorithm}

Algorithm~\ref{algo:PWAunder} is based on the sandwich algorithm \cite{rote1992convergence}, and is demonstrated in Figure \ref{fig:PWAunder}.  At each of $N_p$ evaluation points, $\{(x_p,\hat\Phi_{\mathbf{w}}(x_p)\}_{p=1}^{N_p}$, the algorithm constructs affine terms, and  
stores the affine terms which result in largest positive error close to $\epsilon$.  This is repeated until the break conditions are met (line 8) with a total of $z$ piecewise affine terms.  
We choose an upper bound $\Phi_{\mathbf{w}}(x_{N_p})$ (line 10),
as it is unreasonable to infer the probability of an event beyond
$\underset{j\in\mathbb{N}_{[1,N_s]}}{\max}{(p_i^{\top}\overline{G}\overline{\mathbf{w}}_j)}$, and it assures \eqref{eq:approxineq} holds on $\mathcal{D}_i$.  
This solves Problem~\ref{ps:ecf}.

\subsection{Underapproximative, Conic Optimization Problem}

We replace the individual chance constraints in \eqref{eq:stoc_ecf_cr_ICC} and the lower bounds in \eqref{eq:stoc_ecf_cr_ICC_res} with a conic, convex reformulation, obtained from Algorithm~\ref{algo:PWAunder}, 
resulting in the following.

\begin{subequations}
\label{prob:stoc_ecf_cr_convex}
\begin{alignat}{3}
  \min_{\overline{u},\overline{\delta}} &\hspace{2.5em} \mathbb{E}\left[\left(\overline{{\mathbf{x}}} -\overline{x}_d \right)^\top  Q \left({\overline{\mathbf{x}}} -\overline{x}_d \right) + \overline{u}^\top  R \overline u \right]\label{eq:stoc_ecf_cr_convex_cost}\\
  \mathrm{s.t.} &
      \raisebox{-1\baselineskip}[0pt][0pt]{\hspace{-2.1em}\footnotesize{$\begin{array}{ @{} c }
      \hspace{-0.15em}\footnotesize\forall i \in \mathbb{N}_{[1,l]}\\ \hspace{0.4em}\footnotesize\forall r\in \mathbb{N}_{[1,z^i]}\end{array}$}\hspace{1.7em}}
    \raisebox{-.5\baselineskip}[0pt][0pt]{\hspace{-2.3em}$\left\{\kern-\nulldelimiterspace\begin{array}{ @{} c } \mathstrut \\ \mathstrut \\ \mathstrut \end{array}\right.$}
  \hspace{-0.7em}a_{i,r}(q_i-p_i^{\top}(\overline{A}x_0+\overline{B}\overline{u})) + c_{i,r}\geq 1 - \delta_i \label{eq:stoc_ecf_cr_convex_ICC} \\
  &\hspace{7.5em} q_i-p_i^{\top}(\overline{A}x_0+\overline{B}\overline{u}) \geq\ x^{lb}_{i}\label{eq:stoc_ecf_cr_convex_ICC_res}\\
  &\hspace{6.2em} \sum_{i=1}^{l} \delta_i \leq \Delta,\; \delta_i \geq 0,\ \Delta\in [0,1]\label{eq:stoc_ecf_cr_convex_ra}\\
  & \hspace{15.4em} \overline{u}\in\mathcal{U}^{N}\label{eq:stoc_ecf_cr_convex_u}
\end{alignat}
\end{subequations}

The optimization problem in \eqref{prob:stoc_ecf_cr_convex} can be posed as a second-order cone program \cite[Sec. 4.4]{boyd2004convex}. 
Algorithm~\ref{algo:solvingConvex}
summarizes how the methods described in this section 
solve \eqref{prob:stoc_ecf_cr_convex}.

\begin{algorithm}[h]
  \caption{Underapproximative, conic optimization \eqref{prob:stoc_ecf_cr_convex}}
	\label{algo:solvingConvex}
	Time horizon N, $\Delta T$, polytopic set $\{P,q\}$, samples $\{ \mathbf{w}_j\}_{j=1}^{N_s}$, \eqref{eq:cdfecfwkern}, \eqref{eq:momentecfwkern}, evaluation points $N_p$, desired error $\epsilon$, desired number of affine terms $N_{dr}$, smoothing matrix $\overline{\Sigma}$.
	\\
	\textbf{Output}:
	Open loop input  $\overline{u}$, risk allocation $\overline{\delta}$
	\begin{algorithmic}[1]
        \For{$i\in\mathbb{N}_{[1,l]}$}
        \State Let $\mathcal{D} = [\underset{ j\in\mathbb{N}_{[1,N_s]}}{\min}{(p_i^{\top}\overline{G}\overline{\mathbf{w}}_j)},\underset{ j\in\mathbb{N}_{[1,N_s]}}{\max}{(p_i^{\top}\overline{G}\overline{\mathbf{w}}_j)}]$
        \State $\{(x_p,\Phi_{p_i^{\top}\overline{G}\overline{\mathbf{w}}}(x_p)\}_{p=1}^{N_p}\leftarrow$ Using \eqref{eq:cftocdf} and \eqref{eq:cdfecfwkern}.
        \State $\{(a_{i,r},c_{i,r})\}^{z^i}_{r=1}$ and $x^{lb}_i\leftarrow$ From Algorithm
        \ref{algo:PWAunder}
        \State Let $\mathcal{D}_i \leftarrow [x^{lb}_i,\infty]$
        \EndFor
        \State $\mathbb{E}[\overline{\mathbf{w}}]$, $\mathbb{E}[\overline{\mathbf{w}}^2]\leftarrow$ Using \eqref{eq:momentdesired} and~\eqref{eq:momentecfwkern}.
        \State $C_{\overline{w}}\leftarrow\mathbb{E}[\overline{\mathbf{w}}^2]-(\mathbb{E}[\overline{\mathbf{w}}])^2$
        \State $\{\overline{u},\overline{\delta}\}\leftarrow$ Solve~\eqref{prob:stoc_ecf_cr_convex}.
    \end{algorithmic}%
\end{algorithm}%

\subsection{Convergence and Confidence Intervals}
\begin{figure}
    \centering
    \includegraphics[width=\linewidth]{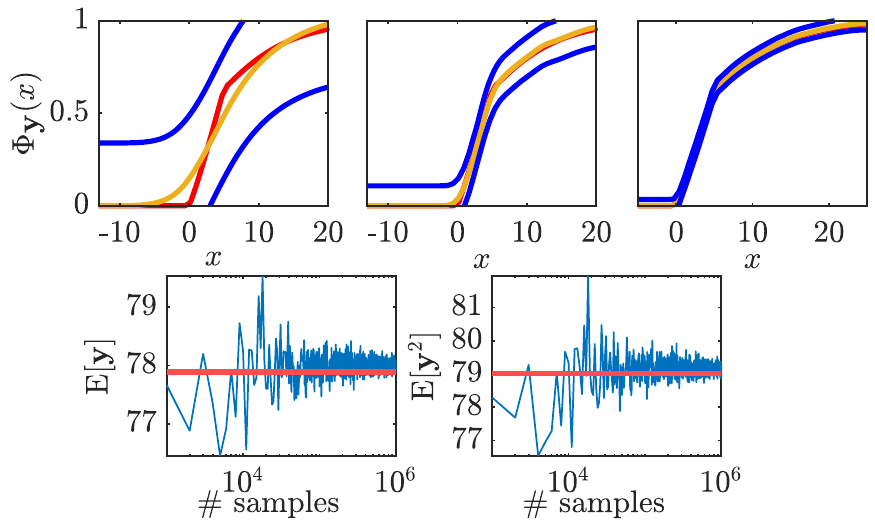}
    \caption{(Top) Approximation $\hat\Phi_{\mathbf{y}}(x)$ (yellow) of $\Phi_{\mathbf{y}}(x)$ (red) with 80\% confidence interval bands (blue) for 10, 100, and 1000 samples. 
    (Bottom) Convergence of $\mathbb{E}[\mathbf{y}]$ and $\mathbb{E}[\mathbf{y}^2]$.
    We presume $\mathbf{y}=\mathbf{f}\mathbf{y}_1+(1-\mathbf{f})\mathbf{y}_2$ for a Bernoulli random variable $\mathbf{f}$, with $\mathbf{y}_1$, $\mathbf{y}_2$, drawn from a gamma distribution $\mathrm{Gam}(k = 2,\theta = 5)$, and 
    a uniform distribution $\mathrm{Unif}[0,5]$, respectively.}
    \label{fig:convergence}
\end{figure}

While~\eqref{prob:stoc_ecf_cr_convex} is convex and conic, its relationship to~\eqref{prob:stoc} is not clear, as it utilizes an under-approximation of the {\em approximate} cumulative distribution function, 
$\hat\Phi_{p_i^{\top}\overline{G}\overline{\mathbf{w}}}(x)$ and approximate moments of $\overline{\mathbf{w}}$.
We first establish asymptotic convergence, then construct confidence intervals to describe a relationship to (\ref{prob:stoc}).
\begin{thm}
If $\hat\upvarphi_{\overline{\mathbf{w}}}(t)$ converges in probability to $\upvarphi_{\overline{\mathbf{w}}}(t)$ as $N_s \rightarrow \infty$, every feasible solution of \eqref{prob:stoc_ecf_cr_convex} is feasible for \eqref{prob:stoc}. 
\end{thm}
\begin{proof}
By~\cite[Thm 2.1]{feuerverger_empirical_1977} $\hat\upvarphi_{\overline{\mathbf{w}}}(t)$ converges to $\upvarphi_{\overline{\mathbf{w}}}(t)$ as $N_s \rightarrow \infty$.
By the Portmanteau theorem, the cumulative distribution function converges~\cite[Thm. 2.1]{billingsley2013convergence}. 
For $\hat\upvarphi_{\overline{\mathbf{w}}}(t)$ 
that is 
differentiable at zero, then by~\eqref{eq:momentecfwkern}, the moments converge~\cite[Thm. 2.3.2]{lukacs_characteristic_1970}. 
\end{proof}
\begin{rem}
The ECF converges at a rate $\sqrt{N_s}$~\cite[Sec. 3]{yu_empirical_2004}.
\end{rem}

Asymptotic convergence establishes the relationship between our convex formulation and the original problem, but it is not practical in order to solve the reformulation quickly nor does it guarantee that~\eqref{eq:stoc_ecf_cr_convex_ICC} is an  under-approximation. 
We provide confidence intervals on the cumulative distribution function, a worst-case under-approximation. 
\begin{defn}[Dvoretzky–Kiefer–Wolfowitz Inequality~\cite{massart1990tight}]\label{def:DKW}
Given an empirical cumulative distribution function, $\hat\Phi^E_{p_i^{\top}\overline{G}\overline{\mathbf{w}}}(x)$, from $N_s$ samples, the probability that the worst deviation is above some $\epsilon_E$ is 
\begin{align}\label{eq:DKW}
    \mathbb{P}\left\{\underset{x\in\mathbb{R}}{\sup}\left(|\hat\Phi^E_{p_i^{\top}\overline{G}\overline{\mathbf{w}}}(x)-\Phi_{p_i^{\top}\overline{G}\overline{\mathbf{w}}}(x)|>\epsilon_E\right)\right\} & \leq \alpha 
\end{align}
for $\alpha = 2e^{-2N_s\epsilon_E^2}$.
\end{defn}
Hence for a desired confidence level $\alpha$, using $N_s$ samples, we have 
$\epsilon_E = ((2N_s)^{-1}\ln{(2/\alpha)})^{1/2}$.  
To make use of (\ref{eq:DKW}) for $\hat \Phi$, we make the following assumption.
\begin{assum}\label{assump:smoothing}
    For $x \in\mathcal{D}_i$, $|\hat\Phi^E_{p_i^{\top}\overline{G}\overline{\mathbf{w}}}(x)-\hat\Phi_{p_i^{\top}\overline{G}\overline{\mathbf{w}}}(x)|\leq \epsilon_D$.
\end{assum}
Assumption~\ref{assump:smoothing} is dependent upon $\overline \Sigma$ and $N_s$, and reasonable for $\Sigma$ chosen to avoid under- or over-smoothing.  Both terms converge to $\Phi_{p_i^{\top}\overline{G}\overline{\mathbf{w}}}(x)$ as $N_s\rightarrow\infty$, so their difference tends to zero \cite[Thm. 20.6]{billingsley_probability_2012}. 

\begin{thm}[Confidence Interval for $\hat\Phi_{p_i^{\top}\overline{G}\overline{\mathbf{w}}}(x)$]\label{thm:confECF}
Given Def.~\ref{def:DKW} and Assumption~\ref{assump:smoothing}, we have that with probability $1 - \alpha$, 
\begin{equation}\label{eq:confECF}
    |\hat\Phi_{p_i^{\top}\overline{G}\overline{\mathbf{w}}}(x)-\Phi_{p_i^{\top}\overline{G}\overline{\mathbf{w}}}(x)|\leq\epsilon_E + \epsilon_D
\end{equation}
\end{thm}
\begin{proof}
For $x \in \mathcal{D}_i$,
by Def.~\ref{def:DKW} and by the least upper bound property~\cite[Def. 5.5.5]{tao2006analysis}, we have that
$        |\hat\Phi^E_{p_i^{\top}\overline{G}\overline{\mathbf{w}}}(x)-\Phi_{p_i^{\top}\overline{G}\overline{\mathbf{w}}}(x)|\leq\epsilon_E
$
is satisfied with probability $1-\alpha$. 
By the properties of absolute value~\cite[Prop. 4.3.3]{tao2006analysis}, 
\begin{equation}\label{eq:proofDKW}
    \hat\Phi^E_{p_i^{\top}\overline{G}\overline{\mathbf{w}}}(x)-\epsilon_E\leq\Phi_{p_i^{\top}\overline{G}\overline{\mathbf{w}}}(x)\leq\hat\Phi^E_{p_i^{\top}\overline{G}\overline{\mathbf{w}}}(x)+\epsilon_E
\end{equation}
By Assumption~\ref{assump:smoothing} and the properties of absolute value, 
\begin{equation}\label{eq:proofAssump}
    \hat\Phi_{p_i^{\top}\overline{G}\overline{\mathbf{w}}}(x)-\epsilon_D\leq\hat\Phi^E_{p_i^{\top}\overline{G}\overline{\mathbf{w}}}(x)\leq\hat\Phi_{p_i^{\top}\overline{G}\overline{\mathbf{w}}}(x)+\epsilon_D
\end{equation}
Since $\hat\Phi_{p_i^{\top}\overline{G}\overline{\mathbf{w}}}(x)$, $\hat\Phi^E_{p_i^{\top}\overline{G}\overline{\mathbf{w}}}(x)$, and $\Phi_{p_i^{\top}\overline{G}\overline{\mathbf{w}}}(x)$ are positive, bounded, right-hand continuous functions~\cite{billingsley_probability_2012},  
we combine~\eqref{eq:proofDKW} and~\eqref{eq:proofAssump}, 
so that
 $   \hat\Phi_{p_i^{\top}\overline{G}\overline{\mathbf{w}}}(x)-\epsilon_E-\epsilon_D\leq\Phi_{p_i^{\top}\overline{G}\overline{\mathbf{w}}}(x)\leq\hat\Phi_{p_i^{\top}\overline{G}\overline{\mathbf{w}}}(x)+\epsilon_E+\epsilon_D
$.  
Thus, we have~\eqref{eq:confECF} by the properties of absolute value. 
\end{proof}
\begin{corr}\label{corr:underapprox}
Given $\hat\Phi^l_{p_i^{\top}\overline{G}\overline{\mathbf{w}}}(x)$, which under-approximates $\hat \Phi_{p_i^{\top}\overline{G}\overline{\mathbf{w}}}(x)$  according to (\ref{eq:approxineq}) on $\mathcal{D}_i$, and the confidence interval $\epsilon_D + \epsilon_E$ in (\ref{eq:confECF}) with likelihood $1-\alpha$, we have
$    \hat\Phi^l_{p_i^{\top}\overline{G}\overline{\mathbf{w}}}(x)-\epsilon-\epsilon_E-\epsilon_D\leq\Phi_{p_i^{\top}\overline{G}\overline{\mathbf{w}}}(x)
$
with likelihood $1-\alpha$.
\end{corr}
\begin{proof}
Follows directly from \eqref{eq:approxineq} and~\eqref{eq:confECF}. 
\end{proof}

Corollary \ref{corr:underapprox} establishes a worst-case under-approximation to the true cumulative distribution function.
 A similar approach can be taken 
 for $\mathbb{E}[\overline{\mathbf{w}}]$ and $\mathbb{E}[\overline{\mathbf{w}}^2]$, using results from \cite{anderson1969confidence} and \cite{romano2002explicit}, respectively.    
However, because the approximate moments are cheap to compute (i.e., 3.22 seconds for $10^6$ samples), numerical approximations can be quite accurate (Figure~\ref{fig:convergence}).  In contrast, the computational cost of sampling is high for the chance constraint under-approximation. 

Algorithm 2 and the optimization reformulation (\ref{prob:stoc_ecf_cr_convex}), along with convergence results and confidence intervals in this section, solve Problem~\ref{ps:optimal_hp}. 

\section{Examples}\label{sec:numerical}%
We demonstrate our approach on two examples.
We presume $N_s = 1000$, $N_p=1000$, $\epsilon = 1\times10^{-3}$, $N_{dr} = 20$, and $\Delta = 0.2$. 
In each case, we compare our method to a mixed-integer particle control approach~\cite{blackmore_probabilistic_2010}, which uses disturbance samples (we chose 50) to compute an open-loop controller.  To do so, we used Monte-Carlo simulation with $10^5$ disturbance sequences.
All computations were done in MATLAB with a 3.80GHz Xeon processor and 32GB of RAM. 
The optimization problems were formulated in CVX~\cite{cvx} and solved with Gurobi~\cite{gurobi}. 
The inversion \eqref{eq:cftocdf} uses CharFunTool \cite{witkovsky_numerical_2016} and system formulations are implemented in SReachTools~\cite{sreachtools}. 
We use \cite{botev2010kernel}, which employs linear diffusion and a plug-in method, to compute $\overline{\Sigma}$. 

\subsection{Double Integrator}

\label{sec:dint}

\begin{figure}[ht!]
    \centering
    \includegraphics[width=\linewidth]{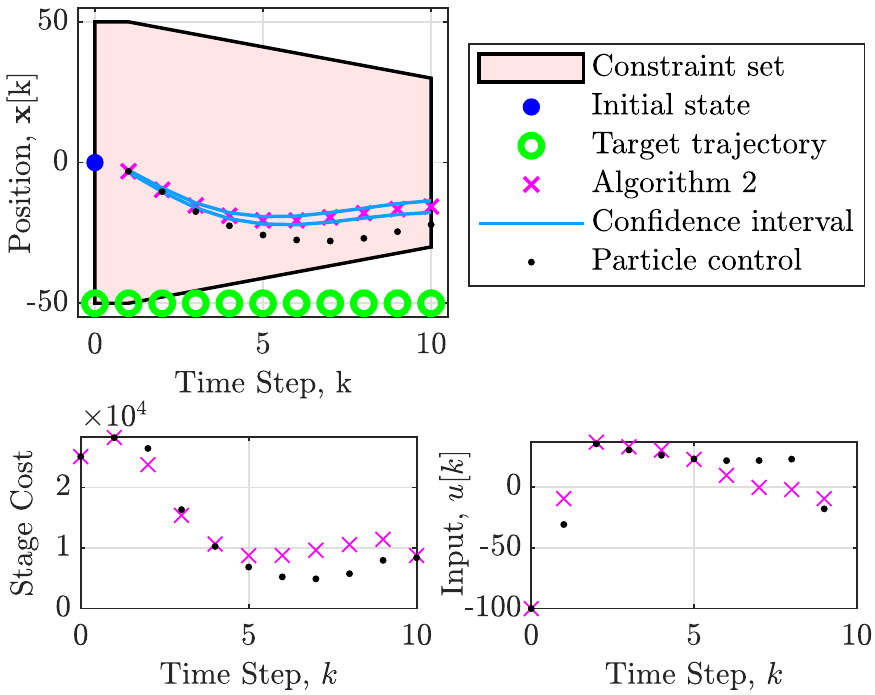}
    \caption{(Top) Mean trajectories for the double integrator. 
     Algorithm~\ref{algo:solvingConvex} satisfies the desired constraint satisfaction likelihood, while particle  control \cite{blackmore_probabilistic_2010} does not. 
    The reference trajectory is chosen to test constraint violation.
    (Bottom) Mean stage cost and control input. Algorithm~\ref{algo:solvingConvex} has higher stage cost due to constraint satisfaction.
    }
    \label{fig:DI}
\end{figure}

Consider a double integrator 
\begin{equation}
    \mathbf{x}[k+1] = 
    \begin{bmatrix}
    1 & \Delta T\\
    0 & 1\\
    \end{bmatrix}\mathbf{x}[k]+
    \begin{bmatrix}
    \frac{\Delta T^2}{2}\\
    \Delta T
    \end{bmatrix}u[k] + \mathbf{w}[k]
\end{equation}
with state $\mathbf{x} \in\mathbb{R}^2$, disturbance $\mathbf{w}\in\mathbb{R}^2$, input $u\in\mathcal{U}=[-100,100]\subset\mathbb{R}$, sampling time $\Delta T= 0.25$, and time horizon $N = 10$. 
Disturbance samples are drawn independently for each dimension, 
from a uniform distribution $\mathrm{Unif}[-5,5]$ on $\mathbf{w}_1$, and from a scaled gamma distribution $0.005\cdot \text{Gam}(k=8,\theta=0.5)$ on $\mathbf{w}_2$.
The cost function has $Q= 10I_{22\times 22}$, $R =
10^{-2} I_{10\times 10}$. 
The time-varying constraint set is $   \mathcal{S} = \left\{ t\in \mathbb{N}_{[0,N]}\times
    \mathbb{R}^2 : p_1 t + q_1 \leq \mathbf{x}_1 \leq p_2 t + q_2\right\}$
with $p_1=-p_2 = -2$, $q_1=-q_2 = -50$. 
The reference trajectory, $x_d = [50\ 0]^{\top}$, was chosen intentionally to be outside of the constraint set, to test constraint violation.

\begin{table}

    \centering
    \caption{Empirical evaluation of the constraint satisfaction likelihood 
    and mean computation time, based on $10^5$ samples. 
    }
    \label{tb:linearmodelresults}
    \begin{tabular}{|l||c|c|c|c|}
    \cline{2-5}
    \multicolumn{1}{c}{} & \multicolumn{2}{|c|}{Algorithm~\ref{algo:solvingConvex}} &\multicolumn{2}{|c|}{Particle Control} \\
    \hline
    Example &
    $1-\Delta$ & Time (s) &
    $1-\Delta$ & Time (s)\\\hline\hline
    Double Integrator & 0.912 & 2.502 & 0.697 & 144.6\\ \hline
    Hypersonic Vehicle  & 0.889 & 5.395 & 0.639 & 31.563\\ \hline
    
    \end{tabular}
\end{table}
While the mean state trajectories from Algorithm~\ref{algo:solvingConvex} and from particle control are similar (Figure~\ref{fig:DI}), the stage cost, i.e. the cost 
at each time, and the control trajectories 
differ. 
Algorithm \ref{algo:solvingConvex} exceeds the constraint satisfaction likelihood of 0.8, while particle control falls well below (Table \ref{tb:linearmodelresults}).  
This is due to the fact that Algorithm~\ref{algo:solvingConvex} is based on 1000 disturbance samples, while particle control is based on only 50 (from inherent undersampling due to computational cost).  The higher cost for Algorithm 2 is incurred because of constraint satisfaction.  

\subsection{One-way Hypersonic Vehicle}
Consider a hypersonic vehicle with longitudinal dynamics
\begin{equation}
\label{eq:hyperdyn}
 \begin{array}{rcl}
    \dot h &= &V\sin(\theta-\alpha)\\
    \dot V &= &\frac{1}{m}(T(\Psi,\alpha)\cos \alpha -D(\alpha,\delta_e))-g\sin(\theta-\alpha)\\
    \dot\alpha &= &\frac{1}{mV}(-T(\Psi,\alpha)\sin \alpha -L)+Q+\frac{g}{V}\cos(\theta-\alpha)\\
    \dot\theta &= &Q\\
    \dot Q &=& M(\alpha,\delta_e,\Psi)/I_{yy}
\end{array}   
\end{equation}
with state $\mathbf{x} = [h\ V \ \alpha \ \theta \ Q]^{\top}$ and input $u = [\Psi\ \delta_e]^{\top}$, that includes fuel-to-air ratio
$\Psi$ and elevator deflection $\delta_e$ \cite{parker2007control}. 
We linearize (\ref{eq:hyperdyn}) about the trim condition, $x_d = [85000 \mbox{ ft}, 7702 \mbox{ ft/s}, 0.026\ \mathrm{rad}, 0.026\ \mathrm{rad}, 0\ \mathrm{rad}]$, which is also the reference trajectory, and $u_d = [0.25, 0.2\ \mathrm{rad}]$, and add a disturbance $\mathbf{w} \in \mathbb R^2$, which affects $\dot h$ and $\dot V$ only, with
$\mathbf{w}_1$, $\mathbf{w}_2$ drawn from a scaled Weibull distribution, $2\cdot \text{Weib}(k=5,\theta=4)$, and 
a gamma distribution, $\text{Gam}(k=5,\theta=1)$, respectively.  We discretize in time with $\Delta T = 0.25$,  $N=10$.  The cost function has 
$Q= 10I_{55\times 55}$ and $R =10^{-2}I_{20 \times 20}$.
The constraint set, $
    \mathcal{S} = \{ t\in \mathbb{N}_{[0,N]}\times
    \mathbb{R}^5 : h\in[85000 \mbox{ ft},85200 \mbox{ ft}],
    V\in[7650 \mbox{ ft/s},7750 \mbox{ ft/s}] \}$, and input constraints $\Psi\in[0.2,1.2]$ $\delta_e\in[-0.26\ \mathrm{rad},0.26\ \mathrm{rad}]$
arise from 
the flight envelope and the operational mode
\cite{hicks1993flight,dalle2011flight,fiorentini2009nonlinear}.
\begin{figure}
    \centering
    \includegraphics[width=\linewidth]{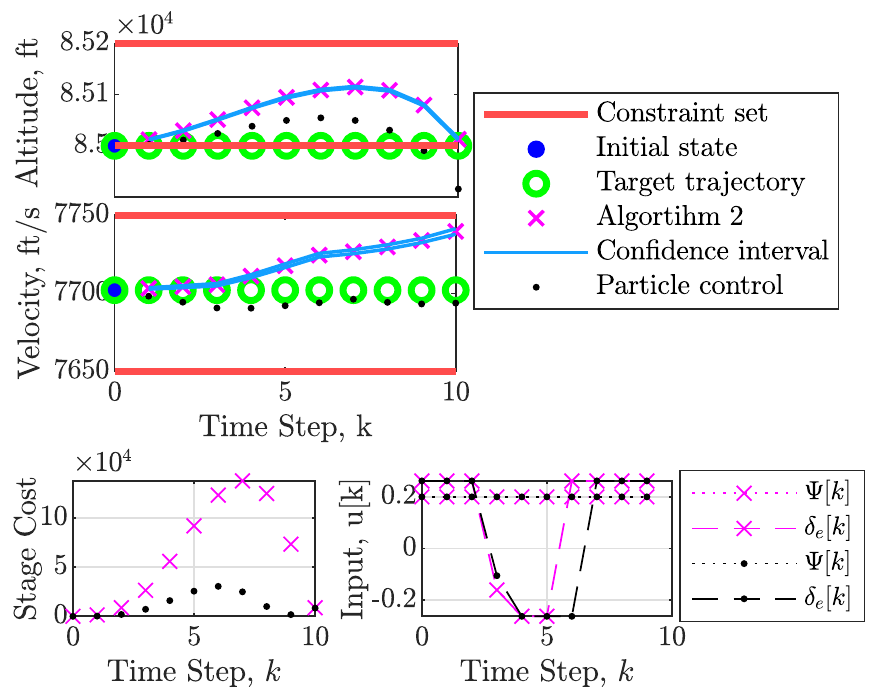}
    \caption{(Top) Mean trajectories for the hypersonic vehicle.  Constraint satisfaction is above the desired likelihood with Algorithm 2, but not with particle control \cite{blackmore_probabilistic_2010}. 
    (Bottom) Mean stage cost and input. 
    The particle control cost is low because constraints are not satisfied.}
    \label{fig:hyp}
\end{figure}


Comparing Algorithm~\ref{algo:solvingConvex} to the particle filter approach, 
mean trajectories (Figure~\ref{fig:hyp}) 
show a similar trend as in Section \ref{sec:dint}.  While constraints are satisfied under Algorithm~\ref{algo:solvingConvex} with at least the desired likelihood, 
particle control violates the altitude constraint, and is excessively conservative with respect to the speed constraint.
The constraint satisfaction likelihood is 0.889 for Algorithm~\ref{algo:solvingConvex}, but only 0.639 for particle control (Table~\ref{tb:linearmodelresults}).


\section{Acknowledgements}

We thank Maria Cristina Pereyra 
and Abraham Vinod for their feedback and discussions.

\bibliographystyle{IEEEtran}
\bibliography{IEEEabrv,shortIEEE,refs}
\end{document}